\documentclass[a4paper,12pt]{amsart}
\usepackage{amsmath, latexsym, amsfonts, amssymb, amsthm, amscd,amstext,amssymb,amsopn,amsthm}

\theoremstyle{plain}
\newtheorem{thm}{Theorem}

\newtheorem{lem}[thm]{Lemma}

\theoremstyle{definition}

\theoremstyle{remark}
\newtheorem*{rem*}{Remark}

\newcommand{\R}{\mathbb{R}}

\renewcommand{\leq}{\leqslant}

\renewcommand{\geq}{\geqslant}
\renewcommand{\ge}{\geq}

\newcommand{\pref}[1]{(\ref{#1})}

\def\({\left(}
\def\){\right)}
\def\[{\left[}
\def\]{\right]}
\def\<{\langle}
\def\>{\rangle}

\title{Regularization inequalities for one-dimensional  Cauchy-type  measures}
\author{Tomasz Byczkowski and Tomasz \.Zak}
\address{\noindent{T. Byczkowski: Institute of Mathematics of Polish Academy of Sciences, Warsaw, Poland}\newline
T. \.Zak: Faculty of Pure and Applied Mathematics, 
  Wroc\l{}aw University of Science and Technology, Wybrze\.ze Wyspia\'nskiego 27, 50-370 Wroc{\l}aw, Poland}
\email{tomasz.zak@pwr.edu.pl}


\thanks{T. Byczkowski was supported by National Science Centre, Poland, grant no. 2015/17/ST1/01233, T. \.Zak was supported by National Science Centre, Poland, grant no. 2015/17/ST1/01043}

\begin{document}
\maketitle
\begin{abstract}
In the paper we investigate various inequalities for the one-dimensional Cauchy measure. We also consider analogous properties for one-dimensional sections of multidimensional isotropic Cauchy measure. The paper is a continuation of our previous investigations \cite{BZ}, where we found, among intervals with fixed measure,  the ones with the extremal measure of the boundary.
Here for the above mentioned measures we investigate inequalities that are analogous to those found for Gaussian measures by Borell in \cite{B} and by Landau and Shepp in \cite{LS}.
\end{abstract}

\section{Introduction}
Gaussian measures occupy central place in various areas of Mathematics. We have
some important and well-known inequalities for these measures: Prekopa-Leindler (\cite{P}), Borell (\cite{B}), 
Ehrhard (\cite{E}) and Landau-Shepp (\cite{LS}). The aim of our research was to find appropriate analogues
of these inequalities for rotationally invariant, standard Cauchy measures. The first step
consisted in examining of the one-dimensional case. Even here the situation is different than 
in the Gaussian case, as half-lines are no longer minimal sets (in the sense of the measure of the boundary). 
It turned out that there are three types of minimal sets, depending on the measures (compare \cite{BZ}).
Further on, we considered one-dimensional sections of $n$-dimensional Cauchy measure
(we call them "Cauchy-type measures") and tried to apply the Steiner-Ehrhard symmetrization
procedure (see \cite{E}), which is the first step in the direction of $n$-dimensional setting.

The classical isoperimetric theorem on the plane states that among all Borel sets with fixed Lebesgue measure the circle has the smallest
perimeter. 
The multidimensional version of the theorem states that in any finite dimension there exists a set with the smallest measure of the boundary and 
this minimum is attained for the ball.
Here by ''the measure of the boundary'' we mean the following: if $A$ is a Borel set and $B_h=\{x\in \R^n:\|x\|<h\}$ we put
$A^h=A+B_h=\{x\in \R^n: \mbox{dist}(x,A)<h\}$. Then the measure of the boundary is equal to 
$
\limsup_{h\to 0^+} \frac{|A^h|-|A|}{h},\ \ \ \ \mbox{where}\ \ \ |A|\ \mbox{denotes the Lebesgue measure of}\ A. $
For simplicity of the language let us call this limit (whenever exists, finite or not), the {\it perimeter} of the set $A$.

The situation is a little different if we consider a probability measure $\mu$ on $\R^n$. This is because of the two reasons.
Firstly, the measure of any Borel set is finite, and, secondly, not only can we look for a set with a minimal measure of the boundary (perimeter),
but we can also seek a set with the maximal perimeter.

\bigskip

Let us start with the definition of the {\it perimeter} in such general situation.
To avoid problems with the existence, we restrict our consideration to convex Borel sets.
Let $A$ be such a set. Put
$$
per(A)=
\limsup_{h\to 0^+} \frac{\mu(A^h)-\mu(A)}{h},
$$
whenever the limit is finite. 

\medskip

Forty years ago mathematicians tried to generalize the isoperimetric theorem. 
Because the Gaussian distribution is one of the most important probability measures,
this problem was investigated first for these distributions. 
It turned out (compare \cite{SC} and \cite{B})  that
among all convex Borel sets in $\R^n$ with the same fixed measure, the half-space i.e. $\{x\in \R^n: x_n>a\}$ has the smallest Gaussian perimeter. 

\medskip

For  convex Borel sets that are {\bf symmetric} with respect to the origin, i.e. such that $-A=A$,
another definition of the perimeter can also be (and is) used. Namely, the authors of \cite{KS} and \cite{LO} put
$$
per(A)=\lim_{\varepsilon \to 0^+} \frac{\mu((1+\varepsilon)A)-\mu(A)}{\varepsilon}.
$$ 
It turned out that for symmetric Gaussian measures the so-called S-hypothesis is valid and a symmetric strip $\{x\in \R^n:|x_n|<a\}$ has the smallest Gaussian perimeter (see \cite{KS},\cite{LO}).

\medskip

During investigation of these isoperimetric properties of Gaussian measures in $\R^n$ many interesting and useful inequalities were found. For instance C. Borell proved 
the following theorem (Theorem 3.1 in \cite{B}), which in our finite-dimensional context can be formulated as below:
\newline
{\it Let $\mu$ be a Gaussian measure in $\R^n$, $A$ a Borel subset of $\R^n$ and let $B$ be the unit ball. Let $\mu(A)=\Phi(\alpha)$, where $\Phi$ is a distribution function of $N(0,1)$. Then for all 
$\varepsilon$>0 there holds} $$\mu( A+\varepsilon B)\ge \Phi(\alpha +\epsilon).$$

\medskip

H.J. Landau and L.A. Shepp proved the following (Theorem 4 in \cite{LS}):
\newline
{\it Let $\mu$ be a Gaussian measure in $\R^n$, $C$ a convex set and let $s$ be any number such that $\mu(C)\ge\Phi(s)$. If $s>0$ then for any $a>$ there holds}
$$\mu(aC)\ge \Phi(as).$$

\medskip

Both above inequalities have very interesting and deep consequences for Gaussian processes (compare \cite{B} and \cite{LS}).
In this paper we examine analogous inequalities for one-dimensional Cauchy and "Cauchy-type" measures.

\section{Cauchy measures}
Standard Cauchy distribution $\mu=\mu_1$ on the real line $\R^1$  has the density function
$$
f(x)=\frac{1}{\pi(1+x^2)},\ \ \  \  x \in \R
$$ 
and rotationally invariant Cauchy distribution $\mu_n$ in $\R^n$ has the one: 
$$
f_n({\bf{x}})=\frac{c_n}{(1+|{\bf{x}}|^2)^{(n+1)/2}},\ \ \ \  {\bf{x}}\in\R^n,\quad
c_n=\frac{\pi^{n/2}}{\Gamma(\frac{n}{2})}
$$

Let $\mu$ be the standard one-dimensional Cauchy measure. For $a<b$ we define $g:=g(a,b)$ by the following equality
\begin{equation}
\label{g}
\mu(-\infty, g) = \mu(a,b)
\end{equation}
$g^{*}:= g^{*}(a,b)$ is defined by the similar identity:
\begin{equation}
\label{g^{*}}
\mu(- g^{*},g^{*}) = \mu(a,b)
\end{equation}

We obtain
\begin{lem} Formulas for $g$ and $g^{*}$ are the following:
\begin{eqnarray*} 
&{}&g(a,b)= -\,\frac{1+a\/b}{b-a},\\
&{}&(g{*})^2(a,b)= \sqrt{1+g^2(a,b)} + g(a,b)
=\frac{\sqrt{1+a^2}\/\sqrt{1+a^2}-1 - a\/b}{b-a}.
\end{eqnarray*}
\end{lem}
\begin{proof}
We have  straightforward computations:
\begin{eqnarray*}
&{}&\mu((a,b))=\int_a^b \frac{d\/t}{\pi(1+t^2)} =\frac{1}{\pi}(\arctan b -\arctan a),\\
&{}& \arctan\,b - \arctan\,a = \frac{\pi}{2} + \arctan\,g(a,b),\\
&{}& \frac{b-a}{1+ab} = \tan\left(\frac{\pi}{2}- \arctan(-g)\right)=\cot (\arctan (- g))=\frac1{-g}.\, 
\end{eqnarray*}
To prove the second formula we obtain
\begin{eqnarray*}
&{}&2\, \arctan \,g^{*}(a,b) = \frac{\pi}{2} + \arctan\,g(a,b)\quad {\text{so}}\\
&{}&\frac{2\,g^{*}(a,b)}{1- (g^{*}(a,b))^2}= -\, \frac{1}{g(a,b)}\,.
\end{eqnarray*}
Solving for $g{*}$ gives
$(g{*})^2(a,b)= \sqrt{1+g^2(a,b)} + g(a,b)$
\begin{equation*}
=\frac{\sqrt{1+a^2}\/\sqrt{1+a^2}-1 - a\/b}{b-a}.
\end{equation*}
\end{proof}

For standard Cauchy measure on $\R^1$ the extremality of intervals or half-lines was explained in \cite{BZ} as follows:
\begin{thm}[{Extremal intervals for Cauchy measure}]
\end{thm}
\begin{itemize}
\item{If  $\mu(a,b)>1/2$ then 
\begin{equation*}
per(-g^{*},g^{*})<per (a,b)< per (-\infty,g)\,.
\end{equation*}
 }
\item{If $\mu(a,b)<1/2$ then
\begin{equation*}
per (-\infty,g)<per (a,b)< per(-g^{*},g^{*}).
\end{equation*}
 }
\item{If $\mu(a,b)=1/2$ (and then $-a=1/b>0$) then
\begin{equation*}
per (-\infty,0)=per (-1/b,b)= per(-1,1)=1/\pi\,.
\end{equation*}
}
\end{itemize}

\subsection{Borell-type inequality}
\begin{thm} [{Borell-type inequality}] For every $a<b$ and every $r>0$ the following holds:
\begin{equation}
\label{adreg1}
g(a-r,b+r)-g(a,b)\geq r/2\,.
\end{equation}
When $\mu(a,b)<1/2$ then
\begin{equation}
\label{adreg2}
g(a-r,b+r)-g(a,b)\geq r
\end{equation}
for all $r>0$ which are small enough. In particular, for $r\leq 2/\sqrt{3}$ the inequality  holds whenever
$\mu(a,b)<1/3$.
\end{thm} 
\begin{proof}
Taking into account the formula \pref{g}   we obtain
\begin{eqnarray*}
&{}&-g(a-r,b+r) + g(a,b) = \frac{1+(a-r)(b+r)}{(b+r)-(a-r)} - \frac{1+a\,b}{b-a}\\
&{=}& -r\,\frac{(b+r)\/b + (a-r)\/a +2}{((b+r)-(a-r))\/(b-a)}\,.
\end{eqnarray*}
After multiplying by $(-1)$ and dividing by $r$ we obtain
\begin{equation*}
 \frac{(b+r)\/b + (a-r)\/a +2}{((b+r)-(a-r))\/(b-a)}\geq \frac{(b+r)\/b + (a-r)\/a }{((b+r)-(a-r))\/(b-a)}\geq \frac{1}{2}.
\end{equation*}
Indeed, we have
\begin{eqnarray*}
&{}& 2\/(b+r)\/b + 2\/(a-r)\/a - ((b+r)-(a-r))\/(b-a) \\
&{=}&2\/(b+r)\/b - (b+r)\/(b-a) + 2\/(a-r)\/a  + (a-r)\/(b-a)\\
&{=}& (b+r) [2\/b - b + a] + (a-r) [2\/a + b - a] = (b+a)^2\geq 0\,.
\end{eqnarray*}
To justify \pref{adreg2} we have to solve the inequality
\begin{equation*}
 \frac{(b+r)\/b + (a-r)\/a +2}{((b+r)-(a-r))\/(b-a)}\geq  1
\end{equation*}
and this is equivalent to the inequality
\begin{equation*}
2\, \frac{1+ a\/b}{b-a}\geq  r
\end{equation*}
or, equivalently, to
\begin{equation*}
g(a,b)\leq -\,  r/2
\end{equation*}
which justifies the first statement of  \pref{adreg2}. For the last part observe that if $r\leq 2/\sqrt{3}$
then $g(a,b) \leq -\, 1/\sqrt{3}$ implies $g(a,b) \leq -\, r/2$ which yields the inequality \pref{adreg2}. 
Inequality $g(a,b) \leq -\, 1/\sqrt{3}$ is, in turn, equivalent to the inequality 
\begin{equation*}
\mu(a,b) \leq \int_{-\infty}^{-1/\sqrt{3}} \frac{d\/t}{1+t^2}= \frac{1}{\pi}(\arctan \frac{-1}{\sqrt{3}} + \frac{\pi}{2}) = \frac{1}{3}\,.
\end{equation*}
\end{proof}
\subsection{Landau-Shepp-type inequality}
\begin{thm}
[{Landau-Shepp-type inequality}]
For every $a<b$ and every $r>0$ the following holds:
\begin{equation*}
g(r\/a,r\/b)\,\geq \, r\,g(a,b) \quad {\text{if and only if}} \quad r\geq 1.
\end{equation*}
\end{thm}
\begin{proof}
Straightforward computation:
\begin{eqnarray*}
&{}&g(r\/a,r\/b)- r\,g(a,b) = -\, \frac{1+r^2 a\/b}{r\/(b-a)} + r\, \frac{1+ a\/b}{b-a}\\
&=&\frac{-1-r^2 a\/b+ r^2 + r^2\/a\/b}{r\/(b-a)} =\frac{ r^2 - 1}{r\/(b-a)}.
\end{eqnarray*}
\end{proof}
\subsection{Concavity of $g(a,b)$}
\begin{thm}
 Function $g(a,b)$ is concave as a function of variables  $a$, $b$, $a<b$
\end{thm}
\begin{proof}
The explicit formulas for second derivatives:
\begin{eqnarray*}
&{}&\frac{\partial^2 g}{\partial a^2 }= -2\,\frac{1+b^2}{(b-a)^3}\,, \quad \frac{\partial^2 g}{\partial b^2 }= -2\,\frac{1+a^2}{(b-a)^3}\,,\\
&{}&\frac{\partial^2 g}{\partial a \partial b }= 2\,\frac{1+a\/b}{(b-a)^3}\,.
\end{eqnarray*}
Computing the determinant of the Hessian  we obtain
\begin{eqnarray*}
&{}& \mbox{det} \,Hess(g)(a,b) = (b-a)^{-6}\,[(1+a^2)(1+b^2) - (1+a\/b)^2]\\
 &=& (b-a)^{-6}\,(a-b)^2\geq 0\,,
\end{eqnarray*}
which, together with $\frac{\partial^2 g}{\partial a^2 }<0$, show that the Hessian is negative-definite.
\end{proof}
\section{One-dimensional sections of multidimensional Cauchy measures} 
We start with an important property of a standard one-dimensional Cauchy measure.
\subsection{Concavity of the function $g(a,b)$}
For a probability density function $f$ we define $g(a,b)$ as a function of intervals $(a,b)$, $-\infty\leq a <b<\infty$ by the following formula
\begin{equation}
\label{g_0}
\int_a^b f(t)\,dt = \int_{-\infty}^{g(a,b)} f(t)\,dt\,.
\end{equation}
We further assume that the function $f$ is differentiable and denote  for simplicity
 \begin{equation}
\label{kappa}
\chi (x) = (1/f(x)) ' \,.
\end{equation}
We have the following
\begin{lem}
Assume that the probability density $f$ is differentiable, decreasing on $(0,\infty)$ and $f(-x)= f(x)$. We also assume that $1/f$ is convex and
denote by $\chi (x) = (1/f(x)) '$ its derivative.  Then the function $g(a,b)$ is concave (as a function of $a, b$, for $a<b$) if and only if the following inequality holds
\begin{equation}
\label{kappaconv}
\frac{\chi(a)\,\chi(b)}{\chi(a) - \chi(b)} \geq \chi(g(a,b))\,.
\end{equation}
\end{lem}
\begin{proof}
Differentiating the defining equality \pref{g_0}
\begin{equation*}
\int_a^b f(t)\,dt = \int_{-\infty}^{g(a,b)} f(t)\,dt\,.
\end{equation*}
we obtain
\begin{eqnarray*}
&{}&\frac{\partial g}{\partial a}= \frac{-\,f(a)}{f(g(a,b))}\,, \quad \frac{\partial g}{\partial b}= \frac{f(b)}{f(g(a,b))}\,,\\
&{}& f(g(a,b)) \frac{\partial^2 g}{\partial a^2}= -\,f'(a) - f'(g(a,b)) \, \frac{f^2 (a)}{f^2 (g(a,b))}\,,\\
&{}& f(g(a,b)) \frac{\partial^2 g}{\partial b^2}= f'(b) - f'(g(a,b)) \, \frac{f^2 (b)}{f^2 (g(a,b))}\,,\\
&{}& f(g(a,b)) \frac{\partial^2 g}{\partial a \partial b}= f'(g(a,b)) \, \frac{f(a)\,f(b)}{f^2 (g(a,b))}\,.
\end{eqnarray*}
We check that the Hessian of the function $g$ is negative definite. 
For $a<b$ we obtain $g(a,b)<b$. The convexity of $1/f$ implies that $-f'(x)/f^2(x)$ is increasing so that
\begin{equation*}
-\, \frac{f'(b)}{f^2 (b)}> -\,\frac{f'(g(a,b))}{f^2 (g(a,b))}\quad {\text{hence}} \quad \frac{\partial^2 g}{\partial b^2}<0\,.
\end{equation*}
Moreover, 
\begin{eqnarray*}
&{}&  f^2(g(a,b))\, \mbox{det} \,Hess(g)(a,b) = \frac{ f'(g(a,b))}{f^2 (g(a,b))}\[f'(a) \, f^2 (b) - f'(b) \, f^2 (a)\] -f'(a)\,f'(b)\,,
\end{eqnarray*}
and the condition for non-negativity of the above expression is equivalent to 
\begin{equation}
\label{conv0}
  \frac{ f'(g(a,b))}{f^2 (g(a,b))}\[\frac{f'(a)}{ f^2 (a)} - \frac{f'(b) }{f^2 (b)}\] \geq 
\frac{f'(a)}{ f^2 (a)}\,\frac{f'(b) }{f^2 (b)}\,.
\end{equation}
Taking into account the definition of the function $\chi$ we rewrite the above inequality as follows:
\begin{equation*}
\chi(g(a,b))\,(\chi(a) - \chi(b)) \geq \chi(a)\,\chi(b)\,.
\end{equation*}
By the requirement that $1/f(x)$ is convex we obtain that $\chi(x) = -f'(x)/f^2(x)$ is increasing so the expression within the bracket
on the left-hand side of the above inequality is negative. Dividing by this expression, we obtain \pref{kappaconv}.
Observe that by the definition we have $\chi(x) = -\, \frac{f'(x)}{f^2(x)}$ and $\chi(-x) = -\,\chi(x)$, $\chi(x) \geq 0$ if $x>0$. 
Therefore, if $a<0<b$ and $g(a,b)<0$ then the left-hand side of \pref{kappaconv} is positive while the right-hand side is negative and the inequality holds automatically. In all the remaining cases we have $\chi(g(a,b))\,\chi(a)\,\chi(b)\leq 0$ and 
\begin{equation*}
\frac{\chi(a) - \chi(b)}{\chi(g(a,b))\,\chi(a)\,\chi(b)}\geq 0\,.
\end{equation*}
Multiplying both sides of \pref{kappaconv}
by this expression we obtain
\begin{equation}
\label{conv1}
\frac{1}{\chi(b)} - \frac{1}{\chi(b)} \leq \frac{1}{\chi(g(a,b))}\,,
\end{equation}
with the exception for the case when $a<0<b$ and simultaneously $g(a,b)<0$. 
\end{proof}
Now we prove analogous property for one-dimensional sections of multidimensional isotropic Cauchy measure.
\begin{thm}
Suppose that $\nu_{\alpha,n}$, $\alpha\geq 0$,  is a probability measure with the density $f_{\alpha,n}$:
\begin{equation}
\label{marginal}
f_{\alpha,n}(x)=\frac{c_n}{(1+\alpha^2+x^2)^{(n+1)/2}}.
\end{equation}
Then the function $g(a,b):=g_{\alpha}(a,b)$ defined by \pref{g_0} is a concave function of two variables $a,\ b$, for $a<b$.
\end{thm}
\begin{proof}
We check that the inequality \pref{kappaconv} holds. We rewrite it in the equivalent form
\begin{equation}
\label{kappaconv1}
\( \frac{1}{\chi(b)} - \frac{1}{\chi(b)} \)^{-1} \geq \chi(g(a,b)\,.
\end{equation}
We first check that the assumptions of the previous lemma are satisfied. We obtain  
\begin{equation*}
\chi(x) = c\,x\,(1+\alpha^2 + x^2)^{\frac{n-1}{2}}\,
\end{equation*}
and  it is clear that all the assumptions are satisfied. We note that  $\lim_{x\to \infty} \chi(x) = \infty$.  \\
First, let us observe that $\lim_{a\to -\,\infty} g(a,b) = g(-\infty,b)=b$ and, at the same time, 
$\lim_{a\to -\, \infty}\chi(a) = -\,\infty$ so we obtain equality in \pref{kappaconv1} for $a= - \, \infty$.  Analogously, 
$\lim_{b\to \infty} g(a,b) = g(a,\infty) = -\,a$. Since $\chi(-a) = -\,\chi(a)$, we also get the equality for $b= \infty$. 
For $a=b$ we have $g(a,a) = -\,\infty$ hence \pref{kappaconv1} obviously holds.
To prove \pref{kappaconv1} in whole generality we use Lagrange method to find extremal values of the function 
\begin{equation*}
F(a,b) + \lambda\, g(a,b) = \frac{1}{\chi(b)} - \frac{1}{\chi(a)} + \lambda \,g(a,b)
\end{equation*}
under the condition $g(a,b)=t$. We obtain
\begin{equation*}
\frac{\partial F(a,b)}{\partial a} + \lambda\, \frac{\partial g(a,b)}{\partial a} = 0\,,
\quad \frac{\partial F(a,b)}{\partial b} + \lambda\, \frac{\partial g(a,b)}{\partial b} = 0\,.
\end{equation*}
Taking into account the form of the first derivatives of $g$ we obtain
\begin{equation*}
\frac{(1/\chi(a))'}{f_{\alpha,n}(a)} = \frac{(1/\chi(b))'}{f_{\alpha,n}(b)} = -\, \frac{\lambda}{f_{\alpha,n}(g(a,b))}\,.
\end{equation*}
By a direct computation we check that the function $\frac{(1/\chi(x))'}{f_{\alpha,n}(x)}$ is injective on $(0,\infty)$. Therefore, extremal values of 
the function $F$ can only be attained at $a=\pm b$. Thus, it is sufficient to check the inequality for $a= -\,b$. \\
Denote $b=-a=p$ and $h(p)=g(-p,p)$. We have to show that for $p>0$ the following holds:
\begin{equation}
\label{kappaconv2}
2\,\chi(h(p)) \leq \chi(p).
\end{equation}
Set $H(z)=\int_{-\infty}^z f_{\alpha,n}(t)\,dt$. By the definition of the value $h(p)$ we obtain
\begin{equation*}
H(h(p)) = \int_{-\infty}^{h(p)} f_{\alpha,n}(t)\,dt = \int_{-p}^p f_{\alpha,n}(t)\,dt = 2\,\int_0^p f_{\alpha,n}(t)\,dt\,.
\end{equation*}
We put $x(p)$ such that
\begin{equation*}
\chi(x(p)) = \frac{1}{2}\chi(p)<\chi(p)\,.
\end{equation*}
We obtain $h(0)=-\,\infty$, $\chi(0)=0$ hence $x(0)=0$. Moreover, $h(p)<p$ and $x(p)<p$. 
We show that the following holds
\begin{equation*}
H(h(p)) - H(x(p))= \int_{-\infty}^{h(p)} f_{\alpha,n}(t)\,dt - \int_{-\infty}^{x(p)} f_{\alpha,n}(t)\,dt\leq 0.
\end{equation*}
The value of the above function at $0$ is $(-1/2)$; at $\infty$ the value is $0$. If we show that the derivative 
is non-negative then this will justify the above statement. \\
Now, since  the inequality \pref{kappaconv2} is invariant with respect to multiplication by non-negative constants we may 
put
\begin{equation*}
\chi(x)=(1/f_{\alpha,n}(x))'= x\,(1+\alpha^2 + x^2)^{\frac{n-1}{2}}
\end{equation*}
By the identity $2\,\chi(x(p))=  \chi(p)$ we obtain
\begin{equation*}
\frac{x(p)}{p} = \frac{1}{2}\( \frac{1+\alpha^2 + p^2}{1+\alpha^2 + x^2(p)} \)^{\frac{n-1}{2}}\,,
\end{equation*}
and
\begin{equation*}
x'(p) = \frac{1}{2}\( \frac{1+\alpha^2 + p^2}{1+\alpha^2 + x^2(p)} \)^{\frac{n-3}{2}}\,
\frac{1+\alpha^2 + n\/p^2}{1+\alpha^2 + n\/x^2(p)}\,. 
\end{equation*}
By the definition of $h(p)$ we obtain
\begin {equation*}
\frac{d}{d\/p} H(h(p)) = \frac{2}{(1+\alpha^2 + p^2)^{\frac{n+1}{2}}}\,.
\end{equation*}
Analogously, taking into account the formula for $x'$ and $x$ we obtain
\begin{eqnarray*}
&{}&\frac{d}{d\/p} H(x(p)) = \frac{x'(p)}{(1+\alpha^2 + x(p)^2)^{\frac{n+1}{2}}}\\
&=& \frac{1}{2}\,\frac{(1+\alpha^2 + p^2)^{\frac{n-3}{2}}}{(1+\alpha^2 + x(p)^2)^{n-1}}\,
\frac{1+\alpha^2 + n\/p^2}{1+\alpha^2 + n\/x^2(p)}\\
&=& 2\,\frac{(1+\alpha^2 + p^2)^{\frac{n-3}{2}}}{(1+\alpha^2 + p^2)^{n-1}}\,\frac{x^2 (p)}{p^2}
\frac{1+\alpha^2 + n\/p^2}{1+\alpha^2 + n\/x^2(p)}\\
&=& 2\,\frac{x^2 (p)}{p^2\,(1+\alpha^2 + p^2)^{\frac{n+1}{2}}}\,
\frac{1+\alpha^2 + n\/p^2}{1+\alpha^2 + n\/x^2(p)}\,.
\end{eqnarray*}
We thus obtain
\begin{eqnarray*}
 \frac{d}{d\/p} \[H(h(p)) -  H(x(p))\] &=& 2\,
\frac{p^2\,(1+\alpha^2 + n\/x^2(p))   -   x^2 (p)\,(1+\alpha^2 + n\/x^2(p))}{p^2\,(1+\alpha^2 + p^2)^{\frac{n+1}{2}}\,(1+\alpha^2 + n\/p^2)}\\
&=& \frac{(1+\alpha^2)\,(p^2 - x^2 (p))}{p^2\,(1+\alpha^2 + p^2)^{\frac{n+1}{2}}\,(1+\alpha^2 + n\/p^2)} >0
\end{eqnarray*}
and since $2\,\chi(x(p))= \chi(p) < 2\,  \chi(p) $ hence $x(p) < p$ . The proof is now complete.
\end{proof}
\subsection{Regularization inequalities}
Now we investigate analogues of Borell and Landau-Shepp inequalities for measures with densities $f_{\alpha,n}$.
\subsubsection{Borell-type inequality}
\begin{thm}
For $a<b$ and every $r>0$ we obtain
\begin{equation}
\label{ad1}
g_{\alpha}(a-r,b+r) - g_{\alpha}(a,b)\,\geq \frac{r}{2^{1/n}}\,,
\end{equation}
where $g_{\alpha}:= g_{\alpha,n}$ is defined by the density $f_{\alpha,n}$.
When $\mu_{\alpha}(a,b)<1/2$ then
\begin{equation}
\label{adreg2}
g_{\alpha}(a-r,b+r)-g(a,b)\geq r
\end{equation}
for all $r>0$ which are small enough.
\end{thm}
\begin{proof}
We first prove 
the differential form of the inequalities:
\begin{equation*}
-\,\frac{\partial\/g_{\alpha}}{\partial\/a} + \frac{\partial\/g_{\alpha}}{\partial\/b}
\,\geq \, \frac{1}{2^{1/n}} \quad 
{\text{or, if $\mu_{\alpha}(a,b)<1/2$:}} \quad
-\,\frac{\partial\/g_{\alpha}}{\partial\/a} + \frac{\partial\/g_{\alpha}}{\partial\/b}
\,\geq \, 1\,.
\end{equation*}
By the form of the partial derivatives of $g_{\alpha}$ we obtain the following form of these inequalities:
\begin{eqnarray}
\label{borelldif1}
&{}&f_{\alpha,n}(a) + f_{\alpha,n}(b)\, \geq \,\frac{1}{2^{1/n}}\, f_{\alpha,n}(g_{\alpha}(a,b))\\
&{}&f_{\alpha,n}(a) + f_{\alpha,n}(b)\, \geq \, \, f_{\alpha,n}(g_{\alpha}(a,b))\,.
\end{eqnarray}
Let $G(a,b)= f_{\alpha,n}(a) + f_{\alpha,n}(b)$. We seek extrema under the condition $g_{\alpha}(a,b)=t$; in the second inequality we assume that $g_{\alpha}(a,b)=t<0$. Using Lagrange method we obtain
\begin{eqnarray*}
&{}& \frac{\partial\/G}{\partial\/a} + \lambda\,\frac{\partial\/g_{\alpha}}{\partial\/a} =0\,, \quad
  \frac{\partial\/G}{\partial\/b} + \lambda\,\frac{\partial\/g_{\alpha}}{\partial\/b} = 0 \quad{\text{or equivalently}}\\
&{}&  f'(a) -\, \lambda \, \frac{f(a)}{f(g_{\alpha}(a,b))}	= 0\,,\quad f'(b) +\, \lambda \, \frac{f(b)}{f(g_{\alpha}(a,b))}	= 0\,.		
\end{eqnarray*}
We thus obtain
\begin{equation*}
\frac{f'(a)}{f(a)} = \frac{\lambda}{f(g_{\alpha}(a,b))} = -\,\frac{f'(b)}{f(b)}\,.
\end{equation*}
Since $\frac{f'(x)}{f(x)}= \frac{-\,(n+1)\,x}{1+\alpha^2 + x^2}$ we obtain
\begin{equation*}
\frac{a}{1+\alpha^2 + a^2} = \frac{-\,b}{1+\alpha^2 + b^2} \quad {\text{equivalently}} \quad
(a+b)\,(1+\alpha^2 + a\/b) =0\,.
\end{equation*}
Now, we proof the first part of the theorem.\\
{\bf{1.}} The case $-a=b=p>0$, $h(p) = g(-p,p)$. Our first inequality reduces to 
\begin{equation*}
h'(p) = 2\,\(\frac{1+\alpha^2+ h^2(p)}{1+\alpha^2 + p^2} \)^{\frac{n+1}{2}}\geq \frac{1}{2^{1/n}}
\end{equation*}
or, equivalently
\begin{equation}
\label{borelldif2}
\frac{1+\alpha^2+ h^2(p)}{1+\alpha^2 + p^2} \geq \frac{1}{2^{2/n}}\,.
\end{equation}
Define $p_1:=p_1(\alpha)$ by the formula
\begin{equation*}
\frac{1+\alpha^2}{1+\alpha^2 + p_1^2} = \frac{1}{2^{2/n}}\,
\end{equation*}
or, more explicitly $ p_1^2 = (1+\alpha^2)\/(2^{2/n} -1)$. Note that for $0<p<p_1$ we obtain
\begin{equation*}
 2\,\(\frac{1+\alpha^2+ h^2(p)}{1+\alpha^2 + p^2} \)^{\frac{n+1}{2}}\,>\,
2\,\(\frac{1+\alpha^2}{1+\alpha^2 + p_1^2} \)^{\frac{n+1}{2}}\,=
\,\frac{1}{2^{1/n}}\,.
\end{equation*}
We thus assume that $p\geq p_1$. Define $z:=z(p)\geq 0$ such that
\begin{equation*}
 2\,\(\frac{1+\alpha^2+ z^2(p)}{1+\alpha^2 + p^2} \)^{\frac{n+1}{2}} =
 \frac{1}{2^{1/n}}\,.
\end{equation*}
It is enough to show that $z(p) \leq h(p)$. We obtain
\begin{equation*}
h'(p) = 2\,\(\frac{1+\alpha^2+ h^2(p)}{1+\alpha^2 + p^2} \)^{\frac{n+1}{2}}
\end{equation*}
so
\begin{equation*}
\frac{h'(p)}{(1+\alpha^2+ h^2(p))^{\frac{n+1}{2}}} = \frac{2}   {(1+\alpha^2 + p^2)^{\frac{n+1}{2}}}\,. 
\end{equation*}
On the other hand, $z'(p) = p/(2^{2/n}\,z(p))$ and, by the definition of $z(p)$ we obtain
\begin{equation*}
\frac{z'(p)}{(1+\alpha^2+ z^2(p))^{\frac{n+1}{2}}} = 
\frac{p}{z(p)}\,\frac{2^{1-\frac{1}{n}}}   {(1+\alpha^2 + p^2)^{\frac{n+1}{2}}}\,. 
\end{equation*}
Therefore, we obtain
\begin{eqnarray*}
&{}&\frac{d}{d\/p}\[ \int_{-\,\infty}^{h(p)} \frac{d\/t}{(1+\alpha^2 + t^2)^{\frac{n+1}{2}}} -\,
 \int_{-\,\infty}^{z(p)} \frac{d\/t}{(1+\alpha^2 + t^2)^{\frac{n+1}{2}}}  \]\\
&=& \frac{h'(p)}{(1+\alpha^2+ h^2(p))^{\frac{n+1}{2}}} - \frac{z'(p)}{(1+\alpha^2+ z^2(p))^{\frac{n+1}{2}}}\\
&=& \frac{2}   {(1+\alpha^2 + p^2)^{\frac{n+1}{2}}} - 
\frac{p}{z(p)}\,\frac{2^{1-\frac{1}{n}}}   {(1+\alpha^2 + p^2)^{\frac{n+1}{2}}}\\
&=& \frac{2}   {(1+\alpha^2 + p^2)^{\frac{n+1}{2}}}\,\( 1 - \frac{1}{2^{1/n}}\,\frac{p}{z(p)}  \)\\
&=& \frac{2}   {(1+\alpha^2 + p^2)^{\frac{n+1}{2}}}\,\frac{2^{2/n}\,z^2(p) - p^2}{2^{1/n}\,z(p)\,(2^{1/n}\,z(p)+p)}<0
\end{eqnarray*}
since $2^{2/n}\,z^2(p) - p^2 = -\,(1+\alpha^2)\,(2^{2/n}-1)<0$. Taking into account that the value of the function under 
differential at $\infty$ is $0$:
\begin{equation*}
 \int_{-\,\infty}^{\infty} \frac{d\/t}{(1+\alpha^2 + t^2)^{\frac{n+1}{2}}} -\,
 \int_{-\,\infty}^{\infty} \frac{d\/t}{(1+\alpha^2 + t^2)^{\frac{n+1}{2}}} =0
\end{equation*}
we obtain that $h(p) \geq z(p)\geq 0$, for $p\geq p_1$, thus ending the proof of the case 1 and showing that
\begin{equation}
\label{ad2}
h'(p) \,\geq \frac{1}{2^{1/n}}\,.
\end{equation}
We note that the above observation also yields 
\begin{equation}
\label{positive_h}
h(p_1)\geq 0 \quad {\text{hence}} \quad \mu_{\alpha}(-p_1,p_1) \geq 1/2\,.
\end{equation}
{\bf{2.}} We now consider the case $a\,b= -\,(1+\alpha^2)$. Put $a=-\,(1+\alpha^2)/b$, $b>0$.
Then the left-hand side of inequality \pref{borelldif1} takes on the following form
\begin{equation*}
f_{\alpha,n}(-\,(1+\alpha^2)/b) + f_{\alpha,n}(b),
\end{equation*}
while the right-hand side is equal to $ f(g_{\alpha,n}((1+\alpha^2)/b,b))$. We multiply both sides of the equation 
\pref{borelldif1} by the constant $(1+\alpha^2)^{(n+1)/2}$ and put $p=b/\sqrt{1+\alpha^2}$, $h(p) = g_{\alpha,n}(-\/1/p,p)$.
Taking into account scaling property of the function $G$, we obtain the following form of our inequality:
\begin{equation}
\label{borelldif3}
(1+p^{n+1})\,\(\frac{1+h^2(p)}{1+p^2}   \)^{\frac{n+1}{2}} \,\geq \, \frac{1}{2^{1/n}}\,.
\end{equation}
Define
\begin{equation*}
\phi(p) = \frac{1+p^{n+1}}{(1+p^2)^{\frac{n+1}{2}}}\,.
\end{equation*}
We obtain
\begin{equation*}
\phi'(p) = \frac{p\,(n+1)}{(1+p^2)^{\frac{n+3}{2}}}\,(p^{n-1} -1)\,.
\end{equation*}
Therefore, $\phi$ is decreasing on $(0,1)$, increasing on $(1,\infty)$ and attains minimum at $1$,
 $\phi(1)- 2/2{(n+1)/2}<1/2^{1/n}$. We observe that the left-hand side of the inequality \pref{borelldif3}
is invariant with respect to the mapping $p\to 1/p$. Therefore, we consider only $p\geq 1$. For such values of $p$ we define 
$y(p)>0$ by the identity
\begin{equation*}
(1+p^{n+1})\,\(\frac{1+y^2(p)}{1+p^2}   \)^{\frac{n+1}{2}} \,= \, \frac{1}{2^{1/n}}\,.
\end{equation*}
Differentiating, we obtain
\begin{equation*}
y'(p) = \frac{p}{y(p)}\,\frac{1+y^2(p)}{1+p^2} \,(1-p^{n-1})  \,.
\end{equation*}
Hence $y(p)$ is decreasing on $(1,\infty)$ while $h(p)$ is increasing since we have:
\begin{equation*}
\frac{h'(p)}{(1+h^2(p))^{n+1}{2}}= \frac{1+p^{n+1}}{(1+p^2)^{\frac{n+1}{2}}}\,.
\end{equation*}
Moreover, for $p=1$ we obtain from the case 1 that $y(1)=z(1)$ so from the monotonicity of $y(p)$ and 
$h(p)$ we obtain
\begin{equation*}
y(p) \leq y(1) = z(1) \leq h(1) \leq h(p)
\end{equation*}
which implies that 
\begin{equation*}
(1+p^{n+1})\,\(\frac{1+h^2(p)}{1+p^2}   \)^{\frac{n+1}{2}}\,\geq \,
(1+p^{n+1})\,\(\frac{1+y^2(p)}{1+p^2}   \)^{\frac{n+1}{2}}=\frac{1}{2^{1/n}}\,,
\end{equation*}
which ends the proof of the case 2. \\
To show the inequality \pref{ad1} we use the concavity of the function $g$. 
Denote by $\psi:=\psi_{a,b}$
\begin{equation*}
\psi_{a,b}(r) = g(a-r,b+r)\,.
\end{equation*}
Function $\psi_{a,b}(r)$ is concave for $r>0$. Consequently, by concavity we obtain
\begin{equation*}
\frac{\psi_{a,b}(r) - \psi_{a,b}(0)}{r}\geq \psi_{a,b} '(r)=\psi_{a-r,b+r} '(0)\,.
\end{equation*}
However, by the expressions for derivatives of the function $g$ and the inequality \pref{borelldif1} we obtain
\begin{equation*}
\psi_{a-r,b+r} '(0) = \frac{f_{\alpha,n}(a-r)}{g(a-r,b+r)} + \frac{f_{\alpha,n}(b+r)}{g(a-r,b+r)}\geq \frac{1}{2^{1/n}}
\end{equation*}
which finally gives \pref{ad1} and ends the proof of the first part of the theorem. \\

To prove the second part, observe that from Lemma 5.1 in \cite{BZ} we obtain for $a\/b= -\,(1+\alpha^2)$ that
$\mu_{\alpha}(a,b) > 1/2$, consequently $g_{\alpha}(a,b)=t>0$ for such pairs $(a,b)$ thus we exclude that 
case from our further considerations. What thus remains is the case $-a=b=p>0$ and, as before, we put
$h(p)= g_{\alpha}(-\/p, p)$. 
We note that our inequality reduces to
\begin{equation*}
h'(p) = 2\,\(\frac{1+\alpha^2+ h^2(p)}{1+\alpha^2 + p^2} \)^{\frac{n+1}{2}}\geq 1\,.
\end{equation*}
or, equivalently
\begin{equation*}
\frac{1}{(1+\alpha^2 + h(p)^2)^{\frac{n+1}{2}}} \leq \frac{2}{(1+\alpha^2 + p^2)^{\frac{n+1}{2}}}\,.
\end{equation*}
However, this means that 
\begin{equation*}
per(-\/p,p)\geq per (-\,\infty, g_{\alpha}(-p,p))
\end{equation*}
and the fundamental Lemma 5.2 in \cite{BZ} proves that the above inequality holds whenever $\mu_{\alpha}(-p,p)<1/2$, thus ending 
the proof of the second part of the theorem in the differential form. The general version can again be obtained 
from the concavity of the function $g_{\alpha}$.   
\end{proof}
\subsubsection{Landau-Shepp-type Inequality}
\begin{thm}
For every $a<b$ and every $\alpha \,\geq\, 0$ the following holds
\begin{equation}
\label{mult1}
g_{\alpha}(r\/a,r\/b) \,\geq r\, g_{\alpha}(a,b) \quad {\text{if and only if}} \quad r\geq 1.
\end{equation}
\end{thm}
\begin{proof}
We write the differential form of the inequality \pref{mult1}. To do this, we rewrite \pref{mult1}
in the form:
\begin{equation*}
\frac{g_{\alpha}(r\/a,r\/b) - g_{\alpha}(a,b)}{r-1}\, \geq \, g_{\alpha}(a,b)
\end{equation*}
and, when $r \to 1$, we obtain
\begin{equation*}
\frac{d\,g_{\alpha}(a,b) }{d\,r} \,\geq \, g_{\alpha}(a,b)\,,
\end{equation*}
or, equivalently,
\begin{equation}
\label{shepp0}
\frac{\partial \,g_{\alpha}(a,b) }{\partial\,a}\,a +
\frac{\partial \,g_{\alpha}(a,b) }{\partial\,b}\,b\, \geq \, g_{\alpha}(a,b)\,.
\end{equation}
Taking into account the form of the partial derivatives of $g_{\alpha}$, we obtain
\begin{equation*}
\frac{-\,f_{\alpha}(a)}{f_{\alpha}(g(a,b))}\,a + \frac{f_{\alpha}(b)}{f_{\alpha}(g(a,b))}\,b \,\geq \, g_{\alpha}(a,b)
\end{equation*}
or, equivalently
\begin{equation}
\label{multdif0}
-\,f_{\alpha}(a)\,a + f_{\alpha}(b)\,b \, \geq \, g_{\alpha}(a,b)\,f_{\alpha}(g(a,b))\,.
\end{equation}
We show that the inequality \pref{multdif0} holds using Lagrange method. 
We put
\begin{equation*}
F(a,b) = -\,f_{\alpha}(a)\,a + f_{\alpha}(b)\,b \, +\,\lambda\, g_{\alpha}(a,b)
\end{equation*}
and obtain
\begin{eqnarray*}
&{}&\frac{\partial\,F(a,b)}{\partial a} = -\,f_{\alpha}'(a)\,a - f_{\alpha}(a)\, +
\,\lambda \frac{ \partial \,g_{\alpha}(a,b)}{\partial \,a} = 0\\
&{}&\frac{\partial\,F(a,b)}{\partial b} = f_{\alpha}'(b)\,b + f_{\alpha}(b)\, +
\,\lambda \frac{ \partial \,g_{\alpha}(a,b)}{\partial \,b} = 0\\
\end{eqnarray*}
Taking again into account the form of partial derivatives of $g_{\alpha}$ we obtain
\begin{eqnarray*}
&{}&  -\,f_{\alpha}'(a)\,a - f_{\alpha}(a)\, -
\,\lambda \frac{f(a)}{ f(g_{\alpha}(a,b))} = 0\\
&{}&f_{\alpha}'(b)\,b + f_{\alpha}(b)\, +
\,\lambda \frac{f(b)}{ f(g_{\alpha}(a,b))} = 0\\
\end{eqnarray*}
which gives
\begin{equation*}
\frac{f_{\alpha}'(a)}{f(a)}\,a = \frac{f_{\alpha}'(b)}{f(b)}\,b\,.
\end{equation*}
Thus, we have obtained
\begin{equation*}
\frac{-\,(n+1)\,a^2}{1+\alpha^2 + a^2} = \frac{-\,(n+1)\,b^2}{1+\alpha^2 + b^2} \quad \Rightarrow a = \pm b\,.
\end{equation*}
We now put $p=-a=b>0$ and $h(p) = g_{\alpha} (-p,p)$ and consider \pref{multdif0} for these values of $a$ and $b$:
\begin{equation*}
2\,f_{\alpha}(p)\,p \,\geq \, f_{\alpha}(h(p))\,h(p)\,.
\end{equation*}
Taking into account the formula for  $h'(p)$ we obtain the equivalent form of the desired inequality:
\begin{equation}
\label{multdif1}
h'(p) \,\geq \, \frac{h(p)}{p}\,.
\end{equation}
We show that the following holds
\begin{equation}
\label{multdif2}
 \frac{h(p)}{p}\, \leq  \, \frac{1}{2^{1/n}}\,.
\end{equation}
In view of the inequality \pref{ad2}, this will end the proof of the theorem.\\
We prove the inequality \pref{multdif2} for the case $f_{0,n}$, in view of the scaling property.
For this purpose, define
\begin{equation*}
\Lambda(p) = 2\,\int_0^p \frac{dt}{(1+t^2)^{\frac{n+1}{2}}} 
\,-\,\int_{-\infty}^{p/2^{1/n}} \frac{dt}{(1+t^2)^{\frac{n+1}{2}}}\,. 
\end{equation*}
We obtain
\begin{equation*}
\Lambda(0) \,<\,0\,, \quad \Lambda(\infty) = 2\,\int_0^{\infty} \frac{dt}{(1+t^2)^{\frac{n+1}{2}}} 
\,-\,\int_{-\infty}^{\infty} \frac{dt}{(1+t^2)^{\frac{n+1}{2}}}=0\,.
\end{equation*}
Moreover,
\begin{eqnarray*}
&{}&\Lambda'(p) = \frac{2}{(1+p^2)^{\frac{n+1}{2}}} - \frac{1}{2^{1/n}} \frac{1}{(1+p^2/2^{2/n})^{\frac{n+1}{2}}}\\
&=& \frac{2}{(1+p^2)^{\frac{n+1}{2}}} - \frac{2}{(2^{2/n}+p^2)^{\frac{n+1}{2}}}\,\geq 0 
\end{eqnarray*}
hence $\Lambda(p)\,\leq\,0$ which means that
\begin{equation*}
\int_{-\infty}^{h(p)}\frac{dt}{(1+t^2)^{\frac{n+1}{2}}}=  
 2\,\int_0^p \frac{dt}{(1+t^2)^{\frac{n+1}{2}}} \,\leq\,
\int_{-\infty}^{p/2^{1/n}} \frac{dt}{(1+t^2)^{\frac{n+1}{2}}}\,, 
\end{equation*}
which proves the inequality  \pref{shepp0}. To finish the proof observe that \pref{shepp0}
holds for all $a,b$, with $a<b$. We rewrite this putting $ra$ in place of $a$ and $rb$ in place 
of $b$ to obtain
\begin{equation*}
\frac{\partial \,g_{\alpha}(r\/a,r\/b) }{\partial\,(r\/a)}\,(r\/a) +
\frac{\partial \,g_{\alpha}(r\/a,r\/b) }{\partial\,(r\/b)}\,(r\/b)\, \geq \, g_{\alpha}(r\/a,r\/b)\,.
\end{equation*}
The above inequality, however, can  in turn be written down as
\begin{equation*}
\frac{d}{d\/r}\[\frac{g_{\alpha}(r\/a,r\/b)}{r}  \]\,\geq\, 0
\end{equation*}
which means that 
\begin{equation*}
\frac{g_{\alpha}(r\/a,r\/b)}{r}  \quad {\text{is increasing as a function of}}\,\,\, r.
\end{equation*}
The proof of the theorem is completed.
\end{proof}
\subsection{Concavity of the function $g_{\bf{y}}(a,b)$}

Let ${\bf{y}} \in {\mathbb{R}}^{n-1}$  and consider the following density
\begin{equation*}
f_{|{\bf{y}}|,n}(t) =\frac{c}{(1+|{\bf{y}}|^2 + t^2)^{\frac{n+1}{2}}}\,, 
 \end{equation*}
being a one-dimensional section of the $n$-dimensional isotropic Cauchy distribution in the direction of ${\bf{y}}$.
We denote this density as $f_{\alpha,n}(t)$ with $\alpha = |{\bf{y}}|$.

As before, for $z_1<z_2$, we define the function $g(z_1,z_2):=g_{\alpha}(z_1,z_2)$ by the identity
\begin{equation*}
\int_{z_1}^{z_2}\frac{dt}{(1+\alpha^2 +t^2)^{\frac{n+1}{2}}}=
\int_{-\,\infty}^{g_{\alpha}(z_1,z_2)}\frac{dt}{(1+\alpha^2 +t^2)^{\frac{n+1}{2}}}\,.
\end{equation*}
By introducing a new variable $u$ by the formula $t=\sqrt{1+\alpha^2}\,u$ we obtain the following important scaling identity for functions $g$:
\begin{equation}
\label{scaling}
g_{\alpha}(z_1,z_2)= \sqrt{1+\alpha^2}\,g_0(\frac{z_1}{\sqrt{1+\alpha^2}},\frac{z_2}{\sqrt{1+\alpha^2}})\,.
\end{equation}
We prove the following
\begin{thm}
The function 
\begin{equation}
\label{concavity1}
{\mathbb{R}}^{n-1} \times {\mathbb{R}}^2 \ni ({\bf{y}},a,b) \to g_{|{\bf{y}}|}(a,b)\,, \quad a<b\,
\end{equation}
is concave, as a function of $(n+1)$ variables, for $a<b$.
\end{thm}
\begin{proof}
We begin by  computing  the derivatives, using the identity \pref{scaling}:
\begin{eqnarray*}
&{}&\frac{\partial g_{\alpha}}{\partial z_1}|_{(z_1,z_2)} = 
\frac{\partial g_{0}}{\partial z_1}|_{(\frac{z_1}{\sqrt{1+\alpha^2}},\frac{z_2}{\sqrt{1+\alpha^2}})};
\quad
\frac{\partial g_{\alpha}}{\partial z_2}|_{(z_1,z_2)} = 
\frac{\partial g_{0}}{\partial z_2}|_{(\frac{z_1}{\sqrt{1+\alpha^2}},\frac{z_2}{\sqrt{1+\alpha^2}})};
\\
&{}&\frac{\partial g_{\alpha}}{\partial \alpha} = 
\frac{\alpha}{1+\alpha^2}[g_{\alpha}-z_1\,\frac{\partial g_{\alpha}}{\partial z_1}
- z_2\,\frac{\partial g_{\alpha}}{\partial z_2}]\,.
\end{eqnarray*}
Differentiating once again with respect to $\alpha$, we obtain
\begin{eqnarray*}
&{}&\frac{\partial^2 g_{\alpha}}{\partial \alpha^2} = 
\frac{1-\alpha^2}{(1+\alpha^2)^2}[g_{\alpha}-z_1\,\frac{\partial g_{\alpha}}{\partial z_1}
- z_2\,\frac{\partial g_{\alpha}}{\partial z_2}]
+ \frac{\alpha}{1+\alpha^2}[\frac{\partial g_{\alpha}}{\partial \alpha} - 
z_1\,\frac{\partial^2 g_{\alpha}}{\partial z_1^2}\,\frac{(-\,z_1\,\alpha)}{(1+\alpha^2)^{3/2}}\\
&-& 
z_1\,\frac{\partial^2 g_{\alpha}}{\partial z_1\,\partial z_2}\,\frac{(-\,z_2\,\alpha)}{(1+\alpha^2)^{3/2}}
- z_2\,\frac{\partial^2 g_{\alpha}}{\partial z_2^2}\,\frac{(-\,z_2\,\alpha)}{(1+\alpha^2)^{3/2}}
-z_2\,\frac{\partial^2 g_{\alpha}}{\partial z_1\,\partial z_2}\,\frac{(-\,z_1\,\alpha)}{(1+\alpha^2)^{3/2}}
]\,.
\end{eqnarray*}
Taking into account the form of $\frac{\partial g_{\alpha}}{\partial \alpha}$ 
we obtain
\begin{equation*}
\frac{\partial^2 g_{\alpha}}{\partial \alpha^2} = \frac{1}{\alpha\,(1+\alpha^2)}\,\frac{\partial g_{\alpha}}{\partial \alpha}
+  \frac{\alpha^2}{(1+\alpha^2)^2}\,[z_1^2\,\frac{\partial^2 g_{\alpha}}{\partial z_1^2} 
+2\,z_1\,z_2\,\frac{\partial^2 g_{\alpha}}{\partial z_1\,\partial z_2}
+z_2^2\,\frac{\partial^2 g_{\alpha}}{\partial z_2^2} 
]\,.
\end{equation*}
The above calculations enable us to write down the 
Hessian of $g_{\alpha}(z_1,z_2)$ as a function of three variables in the following form:
\vskip 15 pt
\begin{math}
\begin{bmatrix} 
\frac{\partial^2 g_{\alpha}}{\partial z_1^2}; & \frac{\partial^2 g_{\alpha}}{\partial z_1\,\partial z_2}; &
\frac{-\,\alpha}{1+\alpha^2}\(z_1\,\frac{\partial^2 g_{\alpha}}{\partial z_1^2} +
  z_2\,\frac{\partial^2 g_{\alpha}}{\partial z_1\,\partial z_2} \) \\
\frac{\partial^2 g_{\alpha}}{\partial z_1\,\partial z_2};  & \frac{\partial^2 g_{\alpha}}{\partial z_2^2};&
\frac{-\,\alpha}{1+\alpha^2}\(z_2\,\frac{\partial^2 g_{\alpha}}{\partial z_2^2} +
  z_1\,\frac{\partial^2 g_{\alpha}}{\partial z_1\,\partial z_2} \) \\
	\frac{-\,\alpha}{1+\alpha^2}\(z_1\,\frac{\partial^2 g_{\alpha}}{\partial z_1^2} +
  z_2\,\frac{\partial^2 g_{\alpha}}{\partial z_1\,\partial z_2} \);&
	\frac{-\,\alpha}{1+\alpha^2}\(z_2\,\frac{\partial^2 g_{\alpha}}{\partial z_2^2} +
  z_1\,\frac{\partial^2 g_{\alpha}}{\partial z_1\,\partial z_2} \);&
	\frac{\partial^2 g_{\alpha}}{\partial {\alpha}^2}\\
 \end{bmatrix}\,.
\end{math}
\vskip 15 pt
We compute the determinant of the above matrix by multiplying the first row by $\frac{\alpha}{1+\alpha^2}\,z_1$
and adding to the third row; analogously, we multiply the second row by $\frac{\alpha}{1+\alpha^2}\,z_2$ and add to the third one. After that  we get the determinant of the matrix
\vskip 15 pt
\hskip 120 pt
\begin{math}
\begin{bmatrix} 
\frac{\partial^2 g_{\alpha}}{\partial z_1^2}; & \frac{\partial^2 g_{\alpha}}{\partial z_1\,\partial z_2}; &
\frac{-\,\alpha}{1+\alpha^2}\(z_1\,\frac{\partial^2 g_{\alpha}}{\partial z_1^2} +
  z_2\,\frac{\partial^2 g_{\alpha}}{\partial z_1\,\partial z_2} \) \\
\frac{\partial^2 g_{\alpha}}{\partial z_1\,\partial z_2};  & \frac{\partial^2 g_{\alpha}}{\partial z_2^2};&
\frac{-\,\alpha}{1+\alpha^2}\(z_2\,\frac{\partial^2 g_{\alpha}}{\partial z_2^2} +
  z_1\,\frac{\partial^2 g_{\alpha}}{\partial z_1\,\partial z_2} \) \\
	0;&
	0;& \frac{1}{\alpha\,(1+\alpha^2)}
	\frac{\partial g_{\alpha}}{\partial {\alpha}}\\
 \end{bmatrix}\,.
\end{math}
\vskip 15 pt
We compute the determinant of the above matrix by developing it with respect to the  third row. This reduces determinant to the product of the determinant of the first $2\times 2$ matrix by the term 
$\frac{1}{\alpha\,(1+\alpha^2)}
	\frac{\partial g_{\alpha}}{\partial {\alpha}}$.
	Since we already know that $g_{\alpha}(z_1,z_2)$ is concave, as a function of $z_1$, $z_2$, everything reduces to the proof that the derivative $ 	\frac{\partial g_{\alpha}}{\partial {\alpha}}$ is negative, that is, that the function
	$g_{\alpha}(z_1,z_2)$ is decreasing, as a function of $\alpha$.
	
This, however, follows from the multiplicative form of the regularization inequality:
\begin{equation*}
g_0(r\/z_1, r\/z_2) \geq r\,g_0(z_1, z_2)\quad {\text{for}} \quad r\geq 1
\end{equation*}
as follows: assume that $0<\alpha_1< \alpha_2$. From the above property and the scaling property \pref{scaling} of the function 
$g$  we obtain
\begin{eqnarray*}
&{}& g_{\alpha_1}(z_1,z_2) = \sqrt{1+\alpha_1^2}\,g_0\(\frac{z_1}{\sqrt{1+\alpha_1^2}}, \frac{z_2}{\sqrt{1+\alpha_1^2}}\)\\
&=&\sqrt{1+\alpha_1^2}\,
g_0\(\frac{z_1}{\sqrt{1+\alpha_2^2}}\,\sqrt{\frac{1+\alpha_2^2}{1+\alpha_1^2}}, \frac{z_2}{\sqrt{1+\alpha_2^2}}\,
\sqrt{\frac{1+\alpha_2^2}{1+\alpha_1^2}}\)\\
&\geq& \sqrt{1+\alpha_1^2}\,\sqrt{\frac{1+\alpha_2^2}{1+\alpha_1^2}}
g_0\(\frac{z_1}{\sqrt{1+\alpha_2^2}}, \frac{z_2}{\sqrt{1+\alpha_2^2}}
\)\\
&=&  \sqrt{1+\alpha_2^2}\,g_0\(\frac{z_1}{\sqrt{1+\alpha_2^2}}, \frac{z_2}{\sqrt{1+\alpha_2^2}}\)
=g_{\alpha_2}(z_1,z_2)\,.
\end{eqnarray*}
The above inequality shows that the function $g_{\alpha}(a,b)$ is concave, as a function of $(\alpha,a,b)$, for $a<b$. 
Since the norm ${\bf{y}} \to |{\bf{y}}|$ is a convex function and $g_{\alpha}(a,b)$ is decreasing as a function of $\alpha$, the theorem follows.
\end{proof}

\end{document}